\documentclass[12pt]{amsart}
\usepackage{ifthen,verbatim}
\usepackage{floatflt,graphicx,graphics}
\usepackage{tikz}
\usepackage{mathrsfs}
\usepackage{amsmath,amsfonts,amssymb,amsthm}
\usepackage{subcaption}

\nonstopmode
\numberwithin{equation}{section}
\theoremstyle{plain}
\newtheorem{theorem}[equation]{Theorem}
\newtheorem{conjecture}[equation]{Conjecture}
\newtheorem{lemma}[equation]{Lemma}
\newtheorem{corollary}[equation]{Corollary}

\newtheorem{proposition}[equation]{Proposition}

\theoremstyle{definition}

\newtheorem{remark}[equation]{Remark}
\newtheorem{nonsec}[equation]{}

\theoremstyle{remark}

\newcommand{\R}{\mathbb{R}}

\newcommand{\B}{\mathbb{B}}

\newcommand{\K}{\mathcal{K}}

\def\Im{\mathop{\rm Im}\nolimits}
\font\fFt=eusm10 
\font\fFa=eusm7  
\font\fFp=eusm5  
\def\K{\mathchoice{\hbox{\,\fFt K}}{\hbox{\,\fFt K}}{\hbox{\,\fFa K}}{\hbox{\,\fFp K}}}

\newcounter{alphabet}

\newcounter{minutes}
\divide\time by 60
\newcounter{hours}
\multiply\time by 60
\addtocounter{minutes}{-\time}

\begin{document}
\bibliographystyle{amsplain}
\title
{Hilbert metric in the unit ball}

\def\thefootnote{}
\footnotetext{
\texttt{\tiny File:~\jobname .tex,
          printed: \number\year-\number\month-\number\day,
          \thehours.\ifnum\theminutes<10{0}\fi\theminutes}
}
\makeatletter\def\thefootnote{\@arabic\c@footnote}\makeatother

\author[O. Rainio]{Oona Rainio}
\author[M. Vuorinen]{Matti Vuorinen}

\keywords{Cayley-Klein metric, Hilbert metric, hyperbolic metric, Klein-Beltrami model, Klein metric}
\subjclass[2010]{Primary 51M10; Secondary 51M16}
\begin{abstract}
The Hilbert metric between two points $x,y$ in a bounded convex domain $G$ is defined as the logarithm of the cross-ratio of $x,y$ and the intersection points of the Euclidean line passing through the points $x,y$ and the boundary of the domain. Here, we study this metric in the case of the unit ball $\mathbb{B}^n$. We present an identity between the Hilbert metric and the hyperbolic metric, give several inequalities for the Hilbert metric, and results related to the inclusion properties of the balls defined in the Hilbert metric. Furthermore, we study the distortion of the Hilbert metric under conformal and quasiregular mappings.
\end{abstract}
\maketitle

\noindent Oona Rainio$^1$, email: \texttt{ormrai@utu.fi}, ORCID: 0000-0002-7775-7656,\newline 
Matti Vuorinen$^1$, email: \texttt{vuorinen@utu.fi}, ORCID: 0000-0002-1734-8228\newline
1: University of Turku, FI-20014 Turku, Finland\\
\textbf{Funding.} O.R.'s research was funded by Finnish Culture Foundation and Magnus Ehrnrooth Foundation.\\
\textbf{Acknowledgements.} The authors are thankful to the referee for their corrections.\\
\textbf{Data availability statement.} Not applicable, no new data was generated.\\
\textbf{Conflict of interest statement.} There is no conflict of interest.


\section{Introduction}

Hyperbolic geometry is an important area of study in mathematics \cite{b98,b00,f19,h,m14}. To describe this geometric system, the Poincar\'e model is most commonly used. It is often defined for the unit disk, but can be easily extended to higher dimensions. Length minimizing segments, geodesics, in hyperbolic geometry are either line segments passing through the origin or arcs of circles orthogonal to the unit ball. The distances defined by these geodesics are measured with the conformally invariant hyperbolic metric. However, the resulting geometric system is not the only model of the hyperbolic geometry of the unit ball.

In the Klein-Beltrami model, geodesics are chords of the unit ball. As Cayley, Klein, and Hilbert all studied this geometric system at the late-19th century, there are several different names for the metric used for measuring the distances on these geodesics, though its definitions are generally overlapping. It can be called Klein, Cayley-Klein, Cayley-Hilbert-Klein, or Klein-Hilbert metric.  

For all distinct points $x$ and $y$ in a bounded convex domain $G\subset\R^n$, the \emph{Hilbert metric} is defined as \cite[Thm 2.1, p. 157]{b99}  
\begin{equation}\label{q_hil}
h_G(x,y)=\log|u,x,y,v|,
\end{equation}
where $u,v$ are the intersection points of the line $L(x,y)$ passing through points $x,y$ and the domain boundary $\partial G$ ordered in such a way that $|u-x|<|u-y|$. See the definition of the cross-ratio from \eqref{q_crs}. If $x=y$, we set $h_G(x,y)=0$. Hilbert \cite{hil} introduced this metric $h_G$ as an extension of the Klein metric for any bounded convex domain $G$. Even though we focus on this article only on the case of the $n$-dimensional unit ball $\B^n$, we study the metric $h_{\B^n}$ defined as above for $G=\B^n$ and call it the Hilbert metric to avoid possible misunderstanding.  

Unlike the hyperbolic metric $\rho_{\B^n}$, the Hilbert metric is not invariant under the M\"obius automorphisms of $\B^n$ as we easily see from the following main
result of this paper.
\begin{theorem}\label{thm_vuo}
For all $x,y\in\B^n$, the following functional identity holds between the Hilbert metric and the hyperbolic metric:
\begin{align*}
{\rm sh}\frac{h_{\B^n}(x,y)}{2}
=\sqrt{1-m_1^2}\,{\rm sh}\frac{\rho_{\B^n}(x,y)}{2},    
\end{align*}
where $m_1$ is the Euclidean distance from the origin to the line $L(x,y)$.
\end{theorem}
We apply this result to compare the Hilbert metric to other metrics and to find several inequalities
for the Hilbert metric. We also prove a distortion theorem for Hilbert metric under M\"obius
transformations.

The structure of this article is as follows. In Section 3, we prove Theorem \ref{thm_vuo} and present a few other basic results related to either the hyperbolic metric or the Hilbert metric. In Section 4, we give several inequalities for the Hilbert metric. In Section 5, we study the ball inclusion between the Hilbert metric, the Euclidean metric, and the hyperbolic metric, and also give exact formula for the balls in the Hilbert metric defined in the unit disk. Finally, in Section 6, we consider the distortion of the Hilbert metric under M\"obius
transformations and formulate one conjecture. At the end of this article, there is also one result about the distortion of the Hilbert metric under quasiregular mappings.

\section{Preliminaries}

Let us first introduce the notations used in this article. 
For $x\in\R^n\setminus\{0\}$, let $x^*=x/|x|^2$. The dot product of two points $x,y\in\R^n$ is denoted by $x\cdot y$. The cross-ratio of four points $u,x,y,v\in\R^n$ is defined as
\begin{align}\label{q_crs}
|u,x,y,v|=\frac{|u-y||x-v|}{|u-x||y-v|}   
\end{align}
and the Hilbert metric is defined as in \eqref{q_hil}. For $x\in\R^2$, $\overline{x}$ is the complex conjugate of $x$.

For two distinct points $x,y\in\R^n$, $L(x,y)$ is the Euclidean line passing through the points $x,y$ and $[x,y]$ is the Euclidean segment with $x,y$ as its end points. For a point $x\in\R^n$ and $r>0$, the open $x$-centered Euclidean ball with radius $r$ is denoted by $B^n(x,r)$ and its sphere is $S^{n-1}(x,r)$. In the special case $x=0$ and $r=1$, we use the simplified notations $\B^n$ and $S^{n-1}$.

The hyperbolic sine, cosine and tangent are denoted by sh, ch, and th, and their inverse functions are arsh, arch, and arth, respectively. The \emph{hyperbolic metric} is defined as \cite[(4.16), p. 55]{hkv}
\begin{align}\label{q_shr}
\text{sh}^2\frac{\rho_{\B^n}(x,y)}{2}=\frac{|x-y|^2}{(1-|x|^2)(1-|y|^2)},\quad x,y\in\B^n.
\end{align}
If $n=2$, this formula can be simplified to
\begin{align}\label{for_lrho}
\text{th}\frac{\rho_{\B^2}(x,y)}{2}=\left|\frac{x-y}{1-x\overline{y}}\right|.
\end{align}
The hyperbolic segment between two points $x,y\in\B^n$ is denoted by $J^*[x,y]$. To denote the hyperbolic ball with a center $x\in\B^n$ and radius $r>0$, we use $B_\rho(x,r)$. Similarly, the notation $B_h(x,r)$ is the corresponding ball in the Hilbert metric $h_{\B^n}$. Note that the metric used to define these balls are defined either in the unit ball $\B^n$ or the unit disk $\B^2$.

It follows from the Riemann mapping theorem that we can use the conformal invariance to define the hyperbolic metric in any simply-connected plane domain $G\subsetneq\R^2$. However, defining the hyperbolic metric in higher dimensions is possible only in special cases. Because of this, several authors have defined substitutes for the hyperbolic metric called \emph{hyperbolic-type metrics} for dimensions $n>2$. One of them is the \emph{distance ratio metric}, introduced by Gehring and Palka \cite{g76}, which is defined for any domain $G\subsetneq\R^n$ as the function $j_G:G\times G\to[0,\infty)$, \cite[p. 685]{chkv}
\begin{align*}
j_G(x,y)=\log\left(1+\frac{|x-y|}{\min\{d_G(x),d_G(y)\}}\right).   
\end{align*}
For a point $x$ in a domain $G\subsetneq\R^n$, the notation $d_G(x)$ means the Euclidean distance $\inf\{|x-z|\,:\,z\in\partial G\}$ between the point $x$ and the domain boundary $\partial G$.

The following inequality holds in the unit ball domain $G=\B^n$:
\begin{lemma}\label{lem_jrhoine}\cite[Lemma 4.9(1), p. 61]{hkv}
For all $x,y\in\B^n$, $j_{\B^n}(x,y)\leq\rho_{\B^n}(x,y)\leq2j_{\B^n}(x,y)$.   
\end{lemma}

\begin{lemma}\label{lem_myinv} 
Suppose that the points $x,y\in\B^n\setminus\{0\}$ are non-collinear with the origin and $|x|\neq|y|$. Let $c$ be the intersection point $L(x,y)\cap L(x^*,y^*)$. Then the inversion $f:\B^n\to\B^n=f(\B^n)$ in the sphere $S^{n-1}(c,\sqrt{|c|^2-1})$ with $f(x)=y$ is given by
\begin{equation}\label{myInv}
f(z)=c+\frac{(|c|^2-1)(z-c)}{|z-c|^2},
\end{equation}
and $f$ maps the chord $L(x,y)\cap\B^n$ onto itself.
\end{lemma}
\begin{proof}
Since the sphere $S^{n-1}(c,\sqrt{|c|^2-1})$ is orthogonal to the unit ball for all $c\in\R^n$ with $|c|>1$, the inversion $f$ preserves the unit disk. The formula for an inversion in the sphere is given in \cite[(2), p. 26]{hkv}. Verifying that $f(x)=y$ is a simple calculation.
\end{proof}

\begin{remark}\label{r_dim}
The distances $h_{\B^n}(x,y)$, $\rho_{\B^n}(x,y)$, and $j_{\B^n}(x,y)$ only depend on how $x,y$ are fixed on the intersection of $\B^n$ and the two-dimensional plane containing these two points and the origin, so we can fix $n=2$ without loss of generality when studying these metrics.
\end{remark}

\section{The Hilbert metric and the hyperbolic metric}

In this section we prove Theorem \ref{thm_vuo}, give another definition for the Hilbert metric in the unit ball (Corollary \ref{cor_csh}) and present the inequality for the Hilbert metric and the hyperbolic metric (Corollary \ref{cor_irhoc}).

\begin{proposition}\label{prop_xyc}
For all $x,y\in\B^n$ such that $|x|=|y|$,
\begin{align*}
h_{\B^n}(x,y)=2\log\left(\frac{\sqrt{4(1-|x|^2)+|x-y|^2}+|x-y|}{\sqrt{4(1-|x|^2)+|x-y|^2}-|x-y|}\right).   
\end{align*}
\end{proposition}
\begin{proof}
Let $q=(x+y)/2$. Now, the distance from the origin to the line $L(x,y)$ is $|q|=\sqrt{|x|^2-|x-y|^2/4}$. Because of the symmetry due to the equality $|x|=|y|$, we have $|u-q|=|v-q|=\sqrt{1-|q|^2}=\sqrt{1-|x|^2+|x-y|^2/4}$ for $u,v\in L(x,y)\cap S^{n-1}$ such that $|x-u|<|y-u|$. It follows that
\begin{align*}
h_{\B^n}(x,y)
&=\log\frac{|u-y||x-v|}{|u-x||y-v|}
=\log\left(\left(\frac{|u-q|+|x-y|/2}{|u-q|-|x-y|/2}\right)^2\right)\\
&=2\log\left(\frac{\sqrt{4(1-|x|^2)+|x-y|^2}+|x-y|}{\sqrt{4(1-|x|^2)+|x-y|^2}-|x-y|}\right)
\end{align*}
\end{proof}

\begin{lemma}\label{lem_finv}
For all $x,y\in\B^n$ with $|x|\neq|y|$, there is an inversion $f:\B^n\to\B^n=f(\B^n)$ such that $|f(x)|=|f(y)|$ and $f(L(x,y)\cap\B^n)=L(x,y)\cap\B^n$, under which the distance $h_{\B^n}(x,y)$ is invariant.  
\end{lemma}
\begin{proof}
Let then $u,v\in L(x,y)\cap S^{n-1}$ so that the points $u,x,y,v$ occur in this order on the line $L(x,y)$. Let $m=(u+v)/2$ and $w$ such a point on $[x,y]$ that $\rho_{\B^n}(x,w)=\rho_{\B^n}(w,y)$. Let $f$ be the inversion in the sphere $S^{n-1}(c,\sqrt{|c|^2-1})$, where $c=L(w,m)\cap L(w^*,m^*)$. It follows from Lemma \ref{lem_myinv} that the hyperbolic segment $J^*[x,y]$ is preserved under the inversion and $f(w)=m$. By the conformal invariance of the hyperbolic metric, it follows that $\rho_{\B^n}(f(x),m)=\rho_{\B^n}(m,f(y))$ and therefore $|f(x)|=|f(y)|$. Since $c\in L(x,y)$, we have $f(L(x,y))=L(x,y)$ but the points $u,x,y,v$ are in the opposite order after the inversion. By the invariance of the cross-ratio under inversions,
\begin{align*}
h_{\B^n}(x,y)&=\log|u,x,y,v|
=\log|f(u),f(x),f(y),f(v)|
=\log|v,f(x),f(y),u|\\
&=h_{\B^n}(f(x),f(y)).
\end{align*}
\end{proof}

\begin{nonsec}
\textbf{Proof of Theorem \ref{thm_vuo}.}
\end{nonsec}
\begin{proof}
We can assume that $|x|=|y|$ without loss of generality because both metrics are invariant under the inversion $f$ of Lemma \ref{lem_finv}. Let $d=|x-y|/2$. By the proof of Proposition \ref{prop_xyc},
\begin{align*}
h_{\B^n}(x,y)=2\log\left(\frac{\sqrt{1-m_1^2}+d}{\sqrt{1-m_1^2}-d}\right).
\end{align*}
If we denote $E=e^{h_{\B^n}(x,y)/2}$,
\begin{align*}
&E=\left(\frac{\sqrt{1-m_1^2}+d}{\sqrt{1-m_1^2}-d}\right)
\quad\Leftrightarrow\quad
d=\frac{E-1}{E+1}\sqrt{1-m_1^2}\\
&\Leftrightarrow\quad
1-d^2-m_1^2=(1-m_1^2)\frac{(E+1)^2-(E-1)^2}{(E+1)^2}
\end{align*}
By \cite[4.16]{hkv},
\begin{align*}
{\rm sh}\frac{\rho_{\B^n}(x,y)}{2}
=\frac{2d}{1-d^2-m_1^2}
=\frac{E^2-1}{2E\sqrt{1-m_1^2}}
=\frac{1}{\sqrt{1-m_1^2}}\,{\rm sh}\frac{h_{\B^n}(x,y)}{2}.
\end{align*}
\end{proof}

\begin{remark}\label{r_findm1}
As shown below in the proof of Corollary \ref{cor_csh}, the distance $m_1$ from the origin to the line $L(x,y)$ in Theorem \ref{thm_vuo} is
\begin{align*}
m_1=\frac{\sqrt{|x|^2|y|^2-(x\cdot y)^2}}{|x-y|}.    
\end{align*}
By \cite[(B11), p. 460]{hkv}, the point closest to the origin on the line $L(x,y)$ for two points $x,y\in\B^2$ is
\begin{align*}
\frac{\overline{x}y-x\overline{y}}{2(\overline{x}-\overline{y})}.   
\end{align*}
Consequently, in the case $n=2$, we can also write $m_1$ as
\begin{align}\label{q_m1f}
m_1=\frac{|\overline{x}y-x\overline{y}|}{2|x-y|}.
\end{align}
\end{remark}

\begin{corollary}\label{cor_csh}
For all $x,y\in\B^n$,
 \begin{align*}
{\rm ch}\frac{h_{\B^n}(x,y)}{2}=\frac{1-x\cdot y}{\sqrt{(1-|x|^2)(1-|y|^2)}} 
\end{align*}    
\end{corollary}
\begin{proof}
Let $m_1$ be the Euclidean distance $m_1$ from the origin to the line $L(x,y)$. If $x$ or $y$ is 0, then $m_1=0$ and the result follows directly from Theorem \ref{thm_vuo} and \eqref{q_shr}. For two distinct points $x,y\in\B^n\setminus\{0\}$, $m_1$ is the height of the triangle with vertices $0,x,y$ and base $[x,y]$. By elementary geometry related to the area of a triangle, we have
\begin{align*}
&\frac{1}{2}m_1|x-y|=\frac{1}{2}\sqrt{|x|^2|y|^2-(x\cdot y)^2}
\quad\Leftrightarrow\quad
m_1=\frac{\sqrt{|x|^2|y|^2-(x\cdot y)^2}}{|x-y|}\\
&\Leftrightarrow\quad
\sqrt{1-m_1^2}=\frac{\sqrt{(1-x\cdot y)^2-(1-|x|^2)(1-|y|^2)}}{|x-y|}.
\end{align*}
From Theorem \ref{thm_vuo} and \eqref{q_shr}, it now follows that
\begin{align*}
{\rm sh}\frac{h_{\B^n}(x,y)}{2}
=\sqrt{1-m_1^2}\,{\rm sh}\frac{\rho_{\B^n}(x,y)}{2}
=\sqrt{\frac{(1-x\cdot y)^2}{(1-|x|^2)(1-|y|^2)}-1}.
\end{align*}
Since
\begin{align*}
{\rm ch}({\rm arsh}\,t)
={\rm ch}(\log(t+\sqrt{t^2+1}))
=\frac{(t+\sqrt{t^2+1})^2+1}{2(t+\sqrt{t^2+1})}
=\sqrt{t^2+1},
\end{align*}
our result follows.
\end{proof}

\begin{remark}
By \cite[(1.4), p. 56]{b99}, the Klein metric $k$ is defined as
\begin{align*}
{\rm ch}\,k_{\B^n}(x,y)=\frac{1-x\cdot y}{\sqrt{(1-|x|^2)(1-|y|^2)}}. 
\end{align*}
In \cite[(1.6), p. 157]{b99}, this definition is erroneously claimed to fulfill $k_{\B^n}(x,y)=\log|u,x,y,v|$ where the points $u$ and $v$ are as in \eqref{q_hil} for $G=\B^n$. This would mean that the metrics $k_{\B^n}$ and $h_{\B^n}$ are equal in the unit ball. However, as can be seen in Corollary \ref{cor_csh}, we need to have the number 2 as a denominator inside the hyperbolic cosine above so that the equality holds.
\end{remark}

\begin{corollary}
For all $x,y\in\B^n$,
\begin{align*}
{\rm ch}\frac{h_{\B^n}(x,y)}{2}=\frac{1-x\cdot y}{|x-y|}\,{\rm sh}\,\frac{\rho_{\B^n}(x,y)}{2}.    
\end{align*}    
\end{corollary}
\begin{proof}
Follows from \eqref{q_shr} and Corollary \ref{cor_csh}.    
\end{proof}

\begin{corollary}\label{cor_irhoc}
For all $x,y\in\B^n$, $h_{\B^n}(x,y)\leq\rho_{\B^n}(x,y)$. The equality here holds if and only if $x,y$ are collinear with the origin. There is no $c\in\R$ such that $\rho_{\B^n}(x,y)\leq c\cdot h_{\B^n}(x,y)$ for all $x,y\in\B^n$. 
\end{corollary}
\begin{proof}
Follows from Theorem \ref{thm_vuo}.
\end{proof}

\section{Inequalities}

In this section, we first give a sharp inequality between the Hilbert metric and the distance ratio metric in Theorem \ref{thm_kji}. Then, in Theorem \ref{thm_1tur}, we present an inequality for the Hilbert metric between two points $x,y\in\B^n$ that can be obtained by rotating the point closer to the origin around the other point. In Theorem \ref{thm_q}, we offer the inequality for the Hilbert metric that follows by rotating the points $x,y\in\B^n$ around their midpoint. This \emph{Euclidean midpoint rotation} was originally formulated for another hyperbolic-type metric called the triangular ratio metric in \cite{sinb}.

\begin{theorem}\label{thm_kji}
For all $x,y\in\B^n$, $h_{\B^n}(x,y)\leq2j_{\B^n}(x,y)$. With the exception of the trivial case where $x=y$ and $h_{\B^n}(x,y)=j_{\B^n}(x,y)=0$, the equality $h_{\B^n}(x,y)=2j_{\B^n}(x,y)$ holds if and only if $x=-y$. There is no $c\in\R$ such that $j_{\B^n}(x,y)\leq c\cdot h_{\B^n}(x,y)$ for all $x,y\in\B^n$.   
\end{theorem}
\begin{proof}
The inequality $h_{\B^n}(x,y)\leq2j_{\B^n}(x,y)$ follows from Lemma \ref{lem_jrhoine} and Corollary \ref{cor_irhoc}. Similarly, it also follows from these results that there cannot be such $c\in\R$ that $j_{\B^n}(x,y)\leq c\cdot h_{\B^n}(x,y)$ for all $x,y\in\B^n$. However, to prove the rest of our theorem, we need to consider the relation between the distances $h_{\B^n}(x,y)$ and $j_{\B^n}(x,y)$ without the hyperbolic metric. 

By Remark \ref{r_dim}, let $n=2$. Because both metrics are invariant under rotations around the origin and reflections in lines that pass through the origin, we can assume that $u=1$ and $\mu={\rm Arg}(v)\in(0,\pi]$ when $u,v\in L(x,y)\cap S^1$ so that $u\neq v$ and $|x-u|<|y-u|$. For some $0<k_0<k_1<1$, $x=1+k_0(e^{\mu i}-1)$ and $y=1+k_1(e^{\mu i}-1)$. See Figure \ref{fig1}. We have now
\begin{align}\label{q_ckfork01}
h_{\B^2}(x,y)=\log\frac{k_1(1-k_0)}{k_0(1-k_1)}    
\end{align}
for all values of $\mu$. Denote
\begin{align*}
W(\mu,k)=\sqrt{k^2+(1-k)^2+2k(1-k)\cos(\mu)}.    
\end{align*}
We have
\begin{align*}
j_{\B^2}(x,y)=\log\left(1+\frac{\sqrt{2}(k_1-k_0)\sqrt{1-\cos(\mu)}}{1-\max\{W(\mu,k_0),W(\mu,k_1)\}}\right). 
\end{align*}
By differentiation,
\begin{align*}
&\frac{\partial}{\partial\cos(\mu)}\left(\frac{\sqrt{1-\cos(\mu)}}{1-W(\mu,k)}\right)\\
&=\frac{-W(\mu,k)+k^2+(1-k)^2+2k(1-k)\cos(\mu)}{2\sqrt{1-\cos(\mu)}W(\mu,k)(1-W(\mu,k))^2}
+\frac{2k(1-k)(1-\cos(\mu))}{2\sqrt{1-\cos(\mu)}W(\mu,k)(1-W(\mu,k))^2}\\
&=\frac{1}{2\sqrt{1-\cos(\mu)}W(\mu,k)(1-W(\mu,k))}.
\end{align*}
for $k\in\{k_0,k_1\}$. Consequently, the distance $j_{\B^2}(x,y)$ is increasing with respect to $\cos(\mu)\in[-1,1)$ or, equivalently, decreasing with respect to $\mu\in(0,\pi]$. 

Fix now $\mu=\pi$ so that the distance $j_{\B^2}(x,y)$ is at minimum with respect to $\mu$. Now,
\begin{align*}
|x-y|=2(k_1-k_0),\quad
|x|=|1-2k_0|,\quad
|1-2k_1|=|y|,
\end{align*}
from which follows that
\begin{align*}
\log\left(1+\frac{2(k_1-k_0)}{1-\max\{|1-2k_0|,|1-2k_1|\}}\right). 
\end{align*}
We have $\max\{|1-2k_0|,|1-2k_1|\}=1-2k_0$ if either $k_0<1/2<k_1$ and $k_0+k_1\leq1$ or $k_0<k_1\leq1/2$, and $\max\{|1-2k_0|,|1-2k_1|\}=2k_1-1$ if either $k_0<1/2<k_1$ and $k_0+k_1>1$ or $1/2\leq k_0<k_1$ instead. By \eqref{q_ckfork01},
\begin{align*}
\frac{h_{\B^2}(x,y)}{j_{\B^2}(x,y)}
=\frac{\log\left(\dfrac{k_1(1-k_0)}{k_0(1-k_1)}\right)}{\max\left\{\log\left(\dfrac{k_1}{k_0}\right),\log\left(\dfrac{1-k_0}{1-k_1}\right)\right\}}
=\frac{\log\left(\dfrac{k_1}{k_0}\right)+\log\left(\dfrac{1-k_0}{1-k_1}\right)}{\max\left\{\log\left(\dfrac{k_1}{k_0}\right),\log\left(\dfrac{1-k_0}{1-k_1}\right)\right\}}
\leq2,
\end{align*}
where the equality holds if and only if $k_1/k_0=(1-k_0)/(1-k_1)$ or, equivalently, $k_0+k_1=1$. If $k_0+k_1=1$, then $|x|=|1-2k_0|=|1-2k_1|=|y|$. Since $\mu=\pi$ and $x\neq y$, it follows from $|x|=|y|$ that $x=-y$.
\end{proof}

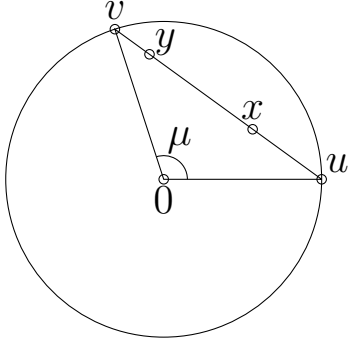
\begin{figure}[ht]
    \centering
    \begin{tikzpicture}[scale=2.1]
    \draw (0,0) circle (1cm);
    \draw (1,0) -- (0,0) -- (-0.309,0.951) -- (1,0);
    \draw (1,0) circle (0.03cm);
    \draw (0.563,0.317) circle (0.03cm);
    \draw (-0.090,0.792) circle (0.03cm);
    \draw (-0.309,0.951) circle (0.03cm);
    \draw (0,0) circle (0.03cm);
    \draw (0.15,0) arc (0:108:0.15);
    \node[scale=1.3] at (1.1,0.1) {$u$};
    \node[scale=1.3] at (0.563,0.44) {$x$};
    \node[scale=1.3] at (0.01,0.87) {$y$};
    \node[scale=1.3] at (-0.309,1.07) {$v$};
    \node[scale=1.3] at (0.1,0.25) {$\mu$};
    \node[scale=1.3] at (0,-0.13) {0};
    \end{tikzpicture}
    \caption{The points $u=1$, $v=e^{\mu i}$, $x=u+k_0(v-u)$, and $y=u+k_1(v-u)$ as in the proof of Theorem \ref{thm_kji}, when $\mu=3\pi/5$, $k_0=1/3$, and $k_1=5/6$.}
    \label{fig1}
\end{figure}

\begin{theorem}\label{thm_1tur}
For all $x,y\in\B^n$,
\begin{align*}
&\log\frac{(1-t^2)\sqrt{|x-y|^2+1-t^2}}{(\sqrt{|x-y|^2+1-t^2}-|x-y|)(1-t^2-|x-y|(\sqrt{|x-y|^2+1-t^2}-|x-y|))}\\
&\leq h_{\B^n}(x,y)\\
&\leq\max\left\{\log\frac{(1+t)(1-t+|x-y|)}{(1-t)(1+t-|x-y|)},
2\log\left(\frac{\sqrt{4(1-t^2)+|x-y|^2}+|x-y|}{\sqrt{4(1-t^2)+|x-y|^2}-|x-y|}\right)\right\},
\end{align*}
where $t=\max\{|x|,|y|\}$.
\end{theorem}
\begin{proof}
By Remark \ref{r_dim}, we can fix $n=2$. By symmetry, we can assume that $|x|=\max\{|x|,|y|\}$. Because $h_{\B^2}(x,y)$ is invariant under rotation around the origin, let us also set $x\in(0,1)$ so that $x=|x|$. Fix $u,v\in L(x,y)\cap S^1$ so that $u\neq v$ and $|x-u|<|y-u|$. Since $L(x,y)=L(x,u)$, rotating $y$ on the closed arc $S^1(x,|x-y|)\cap\overline{\B}^2(0,|x|)$ around $x$ can be done by moving the points $u$ on the unit circle.

Let $\psi={\rm Arg}(u)$ so that
\begin{align}\label{q_uxpsi}
|u-x|=||x|-e^{\psi i}|=\sqrt{1+|x|^2-2|x|\cos(\psi)}.    
\end{align}
Because $y\in L(x,u)\cap\overline{\B^2}(0,|x|)$, as can be seen from Figure \ref{fig2}, we can write $y=x+k(x-u)$ where $k=|x-y|/|u-x|$. We have now
\begin{align*}
|y|&=|x+k(x-u)|=|(1+k)|x|+ke^{\psi i}|=\sqrt{(1+k)^2|x|^2+k^2-2k(1+k)|x|\cos(\psi)}\\
&=\sqrt{(1+|x|^2-2|x|\cos(\psi))k^2+2|x|(|x|-\cos(\psi))k+|x|^2}.   \end{align*}
The condition $|y|\leq|x|$ holds if and only if
\begin{align}
&|y|^2-|x|^2=(1+|x|^2-2|x|\cos(\psi))k^2+2|x|(|x|-\cos(\psi))k\leq0\nonumber\\
&\Leftrightarrow\quad k\leq\frac{2|x|(\cos(\psi)-|x|)}{1+|x|^2-2|x|\cos(\psi)}
\quad\Leftrightarrow\quad |x-y|\leq\frac{2|x|(\cos(\psi)-|x|)}{\sqrt{1+|x|^2-2|x|\cos(\psi)}}\nonumber\\
&\Leftrightarrow\quad \cos(\psi)\geq c_0\equiv\frac{4|x|^2-|x-y|^2+|x-y|\sqrt{4(1-|x|^2)+|x-y|^2}}{4|x|}\label{q_c0l}
\end{align}

Let us then consider the point $v$. We can write $v=x+l(x-u)$ for some $l>0$. We can solve from $|v|=1$ that
\begin{align*}
&|v|^2-1=(1+|x|^2-2|x|\cos(\psi))l^2+2|x|(|x|-\cos(\psi))l-1+|x|^2=0\\
&\Leftrightarrow\quad
l=\frac{-|x|(|x|-\cos(\psi)\pm(1-|x|\cos(\psi)}{1+|x|^2-2|x|\cos(\psi)}.
\end{align*}
Since $l>0$, we need to choose
\begin{align*}
l=\frac{1-|x|^2}{1+|x|^2-2|x|\cos(\psi)}.    
\end{align*}
It follows from \eqref{q_uxpsi} that
\begin{align*}
|v-x|=l|u-x|=\frac{1-|x|^2}{\sqrt{1+|x|^2-2|x|\cos(\psi)}}
=\frac{1-|x|^2}{|u-x|}. 
\end{align*}

We will now have
\begin{align*}
h_{\B^2}(x,y)
&=\log\frac{|u-y||x-v|}{|u-x||y-v|}
=\log\frac{(|u-x|+|x-y|)|v-x|}{|u-x|(|x-v|-|x-y|)}\\
&=\log\frac{(|u-x|+|x-y|)(1-|x|^2)}{|u-x|(1-|x|^2-|x-y||u-x|)}\\
&=\log\left(1+\frac{|x-y|}{|u-x|}\right)-\log\left(1-\frac{|x-y|}{1-|x|^2}|u-x|\right).
\end{align*}
By differentiation,
\begin{align*}
&\frac{\partial}{\partial|u-x|}\left(\log\left(1+\frac{|x-y|}{|u-x|}\right)-\log\left(1-\frac{|x-y|}{1-|x|^2}|u-x|\right)\right)\nonumber\\
&=\frac{|x-y|}{1-|x|^2-|u-x||x-y|}-\frac{|x-y|}{|u-x|(|u-x|+|x-y|)}\geq0\nonumber\\
&\Leftrightarrow\quad
|u-x|\geq-|x-y|+\sqrt{|x-y|^2+1-|x|^2}
\end{align*}
The stationary point above is a minimum, at which the value of the distance $h_{\B^2}(x,y)$ is
\begin{align*}
\log\frac{(1-|x|^2)\sqrt{|x-y|^2+1-|x|^2}}{(\sqrt{|x-y|^2+1-|x|^2}-|x-y|)(1-|x|^2-|x-y|(\sqrt{|x-y|^2+1-|x|^2}-|x-y|))}.    
\end{align*}
Note that $y\in S^1(x,|x-y|)\cap\overline{\B}^2(0,|x|)$ if and only if $1-|x|\leq|u-x|\leq\sqrt{1+|x|^2-2|x|c_0}$ where $c_0$ is as \eqref{q_c0l}. Since $|u-x|=-|x-y|+\sqrt{|x-y|^2+1-|x|^2}$ is not always on the interval $[1-|x|,\sqrt{1+|x|^2-2|x|c_0}]$, the minimum above cannot be always obtained by rotating $y$ on the arc $ S^1(x,|x-y|)\cap\overline{\B}^2(0,|x|)$. However, it is always well-defined since
\begin{align*}
1-|x|^2>|x-y|(\sqrt{|x-y|^2+1-|x|^2}-|x-y|))
\quad\Leftrightarrow\quad
1-|x|^2>-|x-y|^2
\end{align*}
and can therefore be used as a lower limit for $h_{\B^2}(x,y)$. The distance $h_{\B^2}(x,y)$ is at maximum with respect to $y\in S^1(x,|x-y|)\cap\overline{\B}^2(0,|x|)$ when either $|u-x|=1-|x|$ or $|u-x|=\sqrt{1+|x|^2-2|x|c_0}$. If $|u-x|=1-|x|$,
\begin{align*}
h_{\B^2}(x,y)=\log\frac{(1+|x|)(1-|x|+|x-y|)}{(1-|x|)(1+|x|-|x-y|)}. 
\end{align*}
If $|u-x|=\sqrt{1+|x|^2-2|x|c_0}$, then $|x|=|y|$ and $h_{\B^2}(x,y)$ is given by Proposition \ref{prop_xyc}.
\end{proof}

\begin{figure}[ht]
    \centering
    \begin{tikzpicture}[scale=2.1]
    \draw (0,0) circle (1cm);
    \draw (0.866,0.5) -- (-0.214,-0.976);
    \draw (0.866,0.5) circle (0.03cm);
    \draw (0.5,0) circle (0.03cm);
    \draw (0.303,-0.268) circle (0.03cm);
    \draw (-0.214,-0.976) circle (0.03cm);
    \draw (0,0) circle (0.03cm);
    \draw (0.866,0.5) -- (0,0) -- (0.5,0);
    \draw (0.388,0.314) arc (109.4:250.5:0.333);
    \draw[dashed] (0.388,-0.314) arc (-109.4:109.4:0.333);
    \draw[dashed] (0,0) circle (0.5cm);
    \draw (0.15,0) arc (0:30:0.15);
    \node[scale=1.3] at (0.9,0.63) {$u$};
    \node[scale=1.3] at (0.63,0) {$x$};
    \node[scale=1.3] at (0.3,-0.13) {$y$};
    \node[scale=1.3] at (-0.214,-0.83) {$v$};
    \node[scale=1.3] at (-0.13,0) {0};
    \node[scale=1.3] at (0.05,0.17) {$\psi$};
    \end{tikzpicture}
    \caption{The points $x,y,u,v$ as in the proof of Theorem \ref{thm_1tur}, when $x=1/2$, $\psi={\rm Arg}(u)=\pi/6$, and $|x-y|=1/3$.}
    \label{fig2}
\end{figure}
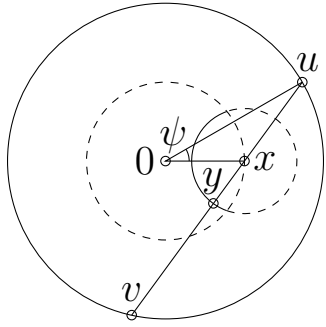

\begin{remark}
The following result related to the Euclidean midpoint rotation also holds for the triangular ratio metric \cite[Thm 5.11, p. 22 \& Thm 5.12, p. 23]{sinb} and the hyperbolic metric \cite[Rmk 4.4, p. 8 of 14]{ird}, and this result could very easily be proven for the distance ratio metric, too. 
\end{remark}

\begin{theorem}\label{thm_q}
For all $x,y\in\B^n$, the distance $h_{\B^n}(x,y)$ is decreasing under such a rotation around the point $(x+y)/2$ that increases the angle $\nu\in[0,\pi/2]$ between the lines $L(x,y)$ and $L(0,(x+y)/2)$ or, if $(x+y)/2=0$, invariant under this rotation. The inequality
\begin{align*}
2\log\left(\frac{\sqrt{4-|x+y|^2}+|x-y|}{\sqrt{4-|x+y|^2}-|x-y|}\right)\leq h_{\B^n}(x,y)   
\end{align*}
holds for all $x,y\in\B^n$ and the equality holds here if and only if $|x|=|y|$. Furthermore, if $|x-y|<2-|x+y|$, then
\begin{align*}
h_{\B^n}(x,y)\leq\log\left(\frac{(2+|x-y|)^2-|x+y|^2}{(2-|x-y|)^2-|x+y|^2}\right),   
\end{align*}
where the equality holds if and only if $x,y$ are collinear with the origin.
\end{theorem}
\begin{proof}
Let $n=2$ by Remark \ref{r_dim}. Fix $q=(x+y)/2$ and $d=|x-y|/2$. By the invariance of $h_{\B^2}(x,y)$ under rotation around the origin, we can assume that $q\in[0,1)$. Fix $\nu$ so that $x=|q|+de^{\nu i}$ and $y=|q|-de^{\nu i}$ as in Figure \ref{fig3}. Because the distance $h_{\B^2}(x,y)$ is symmetric with respect to $x$ and $y$ and invariant under reflection over a line passing through the origin, we can assume that $\nu\in[0,\pi/2]$ without loss of generality. 

As proved in \cite[Rmk 4.4, p. 8 of 14]{ird}, the hyperbolic metric $\rho_{\B^2}(x,y)$ is decreasing with respect to $\nu$. Trivially, the distance $m_1$ from the origin to the line $L(x,y)$ is increasing with respect to $\nu$. By Theorem \ref{thm_vuo}, the distance $h_{\B^2}(x,y)$ is therefore decreasing with respect to $\nu$. 

Let $u,v\in L(x,y)\cap S^{n-1}$ so that $u\neq v$ and $|x-u|<|y-u|$. If $\nu=0$, we have
\begin{align*}
|q-u|=1-|q|,\quad
|q-v|=1+|q|,\quad
h_{\B^2}(x,y)=\log\frac{(1+d)^2-|q|^2}{(1-d)^2-|q|^2},
\end{align*}
and, if $\nu=\pi/2$, then
\begin{align*}
|q-u|=|q-v|=\sqrt{1-|q|^2},\quad
h_{\B^2}(x,y)=2\log\frac{\sqrt{1-|q|^2}+d}{\sqrt{1-|q|^2}-d}.
\end{align*}
Note that the distance $h_{\B^2}(x,y)$ is not defined for $\nu=0$ if $d\geq1-|q|$. Clearly, if we rotate $x$ and $y$ around their midpoint so that the angle $\nu$ becomes 0, the point $x$ stays inside the unit disk if and only if $d<1-|q|$ or, equivalently, $|x-y|<2-|x+y|$. Since $x,y$ stay in $\B^2$ in the rotation to the other direction so that $\nu=\pi/2$, the distance $h_{\B^2}(x,y)$ is always well-defined for $\nu=\pi/2$. 
\end{proof}

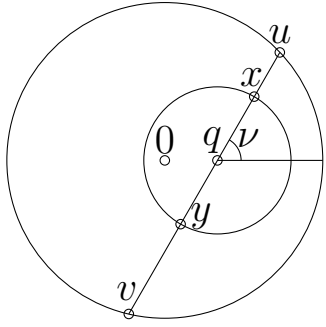
\begin{figure}[ht]
    \centering
    \begin{tikzpicture}[scale=2.1]
    \draw (0,0) circle (1cm);
    \draw (0.728,0.684) -- (-0.228,-0.973);
    \draw (0.333,0) circle (0.03cm);
    \draw (0.728,0.684) circle (0.03cm);
    \draw (0.566,0.404) circle (0.03cm);
    \draw (0.1,-0.404) circle (0.03cm);
    \draw (-0.228,-0.973) circle (0.03cm);
    \draw (0.333,0) circle (0.466cm);
    \draw (0.483,0) arc (0:60:0.15);
    \draw (0.333,0) -- (1,0);
    \draw (0,0) circle (0.03cm);
    \node[scale=1.3] at (0.73,0.8) {$u$};
    \node[scale=1.3] at (0.55,0.53) {$x$};
    \node[scale=1.3] at (0.3,0.13) {$q$};
    \node[scale=1.3] at (0.23,-0.35) {$y$};
    \node[scale=1.3] at (-0.24,-0.83) {$v$};
    \node[scale=1.3] at (0,0.13) {0};
    \node[scale=1.3] at (0.53,0.14) {$\nu$};
    \end{tikzpicture}
    \caption{The points $x,y,q,u,v$ and circle $S^1(q,d)$ as in the proof of Theorem \ref{thm_q} when $q=1/3$, $d=7/15$, $\nu=\pi/3$.}
    \label{fig3}
\end{figure}

\section{Ball inclusion}

In this section, we first study the ball inclusion between the Hilbert metric and the Euclidean metric. In Lemma \ref{lem_hilb}, we offer a formula that can be used for instance writing a code for drawing any ball in the Hilbert metric defined in the unit disk. At the end of this section, there is also a result about the ball inclusion between the Hilbert metric and the hyperbolic metric.

\begin{lemma}\label{lem_inc}
For all $x\in\B^n$ and $0<r<1-|x|$, $B_h(x,l_0)\subseteq B^n(x,r)\subseteq B_h(x,l_1)\subsetneq\B^n$ if and only if
\begin{align*}
l_0&\leq
\begin{cases}
\log\dfrac{(1+|x|)(1-|x|+r)}{(1-|x|)(1+|x|-r)}&\quad\text{if}\quad|x|\leq r,\\
\log\dfrac{(1-|x|^2)\sqrt{r^2+1-|x|^2}}{(\sqrt{r^2+1-|x|^2}-r)(1-|x|^2-r(\sqrt{r^2+1-|x|^2}-r))}&\quad\text{if}\quad|x|>r,
\end{cases}
\\
l_1&\geq\log\frac{(1-|x|)(1+|x|+r)}{(1+|x|)(1-|x|-r)}.
\end{align*}
\end{lemma}
\begin{proof}
Fix $n=2$ by Remark \ref{r_dim}. Let us consider the distance $h_{\B^2}(x,y)$ for $y\in S^1(x,r)$. Let $u\in L(x,y)\cap S^1$ so that $u\neq v$ and $|x-u|<|y-u|$. Recall the differentiation of $h_{\B^2}(x,y)$ with respect to $|x-y|$ from the proof of Theorem \ref{thm_1tur}. It follows that $h_{\B^2}(x,y)$ has a local minimum when $|u-x|=-r+\sqrt{r^2+1-|x|}$. Note that the end points of the interval of $|u-x|$ are now $1-|x|$ and $1+|x|$ since there is no limitation $|y|\leq|x|$. Because
\begin{align*}
&-r+\sqrt{r^2+1-|x|}\leq1-|x|
\quad\Leftrightarrow\quad
|x|\leq r,\\
&-r+\sqrt{r^2+1-|x|}\leq1+|x|
\quad\Leftrightarrow\quad
(3|x|+2r)(1+|x|)\geq0,
\end{align*}
the minimum of $h_{\B^2}(x,y)$ for $y\in S^1(x,r)$ is found either when $|u-x|=1-|x|$ or $|u-x|=-r+\sqrt{r^2+1-|x|}$, depending if $|x|\leq r$ or not, and the maximum is attained when $|u-x|=1+|x|$.
\end{proof}

\begin{corollary}
For all $x\in\B^n$ and $l>0$, $B^n(x,r_0)\subseteq B_h(x,l)\subseteq B^n(x,r_1)$ if and only if
\begin{align*}
r_0\leq\frac{(e^l-1)(1-|x|^2)}{1-|x|+e^l(1+|x|)}\quad\text{and}\quad
r_1\geq\begin{cases}
\dfrac{(e^l-1)(1-|x|^2)}{1+|x|+e^l(1-|x|)}&\quad\text{if}\quad|x|\leq r,\\
\vspace{0.1pt}\\
\dfrac{1}{2}\sqrt{\dfrac{1-|x|^2}{e^l}}(e^l-1)&\quad\text{if}\quad|x|>r.
\end{cases}
\end{align*}
\end{corollary}
\begin{proof}
Follows from Lemma \ref{lem_inc}.   
\end{proof}

\begin{lemma}\label{lem_hilb}
For all $x\in\B^2$ and $l>0$,
\begin{align*}
&B_h(x,l)=\left.\left\{y=x+k(x-\frac{x}{|x|}e^{\psi i})\,\right|\,\psi\in[0,2\pi),\quad0\leq k<k_1(\psi)\vphantom{\frac{1}{1}}\right\},\\
&k_1(\psi)=\frac{(1-|x|^2)(e^l-1)}{1-|x|^2+e^l(1+|x|^2-2|x|\cos(\psi))}
\end{align*}
\end{lemma}
\begin{proof}
Suppose that $y\in S_h(x,l)$. Fix $u,v\in L(x,y)\cap S^{n-1}$ so that $u\neq v$ and $|x-u|<|y-u|$. Now, $y=x+k_1(x-u)$ with $k_1=|x-y|/|u-x|$. If $\psi\in[0,2\pi)$ is chosen so that $u=xe^{\psi}/|x|$, then $|u-x|$ is as in \eqref{q_uxpsi} and, by the proof of Theorem \ref{thm_1tur},
\begin{align*}
h_{\B^2}(x,y)=\log\frac{(1+k_1)(1-|x|^2)}{1-|x|^2-k_1|u-x|^2}.    
\end{align*} 
We can solve that $h_{\B^2}(x,y)=l$ if and only if $k_1$ is as $k_1(\psi)$ in the lemma, from which the result follows.
\end{proof}

Figure \ref{fig4} shows a ball in the Hilbert metric drawn with an R-code that utilizes the result of Lemma \ref{lem_hilb}. 

\begin{figure}
    \centering
    \includegraphics[scale=0.5,trim={1cm 2.5cm 1cm 3cm},clip]{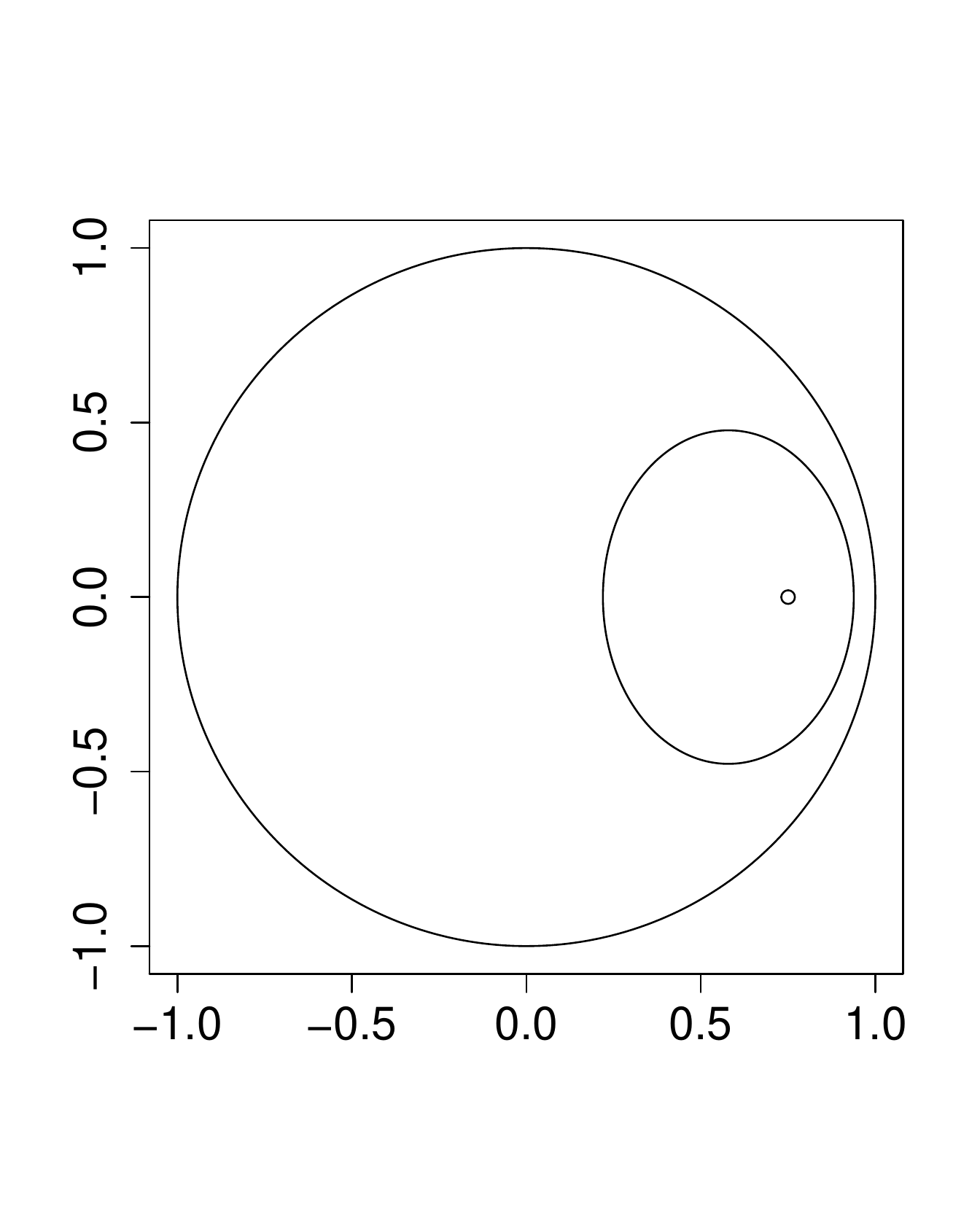}
    \caption{The disk $B_h(x,l)$ drawn with code using Lemma \ref{lem_hilb} in the unit disk for $x=0.75$ and $l=1.5$.}
    \label{fig4}
\end{figure}

\begin{lemma}
For $x\in\B^n$ and $r>0$, $B_h(x,l_0)\subseteq B_\rho(x,r)\subseteq B_h(x,l_1)\subsetneq\B^n$ if and only if
\begin{align*}
l_0\leq2\,{\rm arsh}\left(\sqrt{1-|x|^2}\,{\rm sh}(r/2)\right)\quad\text{and}\quad
l_1\geq r.  
\end{align*}    
\end{lemma}
\begin{proof}
Suppose that $x\neq0$. It follows by Theorem \ref{thm_vuo} that, for all $y\in S_\rho(x,r)$,
\begin{align*}
\sqrt{1-|x|^2}\,{\rm sh}(r/2)\leq{\rm sh}\frac{h_{\B^n}(x,y)}{2}\leq{\rm sh}(r/2),    
\end{align*}
where the equality holds for the first part of the inequality if and only if the line $L(x,y)$ is perpendicular to the line $L(x,0)$ and for the second part if and only if $y\in L(0,x)$. In the special case $x=0$, the equalities of both parts of the inequality hold. The inclusion result follows.   
\end{proof}

\section{Distortion under mappings}

In this section, we use Theorem \ref{thm_vuo} to study the distortion of the Hilbert metric under the following sense-preserving M\"obius transformation. For $a\in\B^n\setminus\{0\}$, define $T_a:\B^n\to\B^n$ as
\begin{align*}
T_a(z)=p_a\circ\sigma_a,    
\end{align*}
where $p_a$ is the reflection in the $(n-1)$-dimensional plane through the origin and orthogonal to $a$, and $\sigma_a$ is the inversion in the sphere $S^{n-1}(a^*,\sqrt{1/|a|^2-1})$ \cite[p. 11]{hkv}. If $n=2$, we have by \cite[p. 459]{hkv}
\begin{align*}
T_a(z)=\frac{z-a}{1-\overline{a}z}.    
\end{align*}

\begin{corollary}\label{cor_dist}
For all $x,y,a\in\B^n$,
\begin{align*}
h_{\B^n}(T_a(x),T_a(y))
=2\,{\rm arsh}\left(\sqrt{\frac{1-m_2^2}{1-m_1^2}}\,{\rm sh}\frac{h_{\B^n}(x,y)}{2}\right),
\end{align*}
where $m_1$ and $m_2$ are the Euclidean distances from the origin to the line $L(x,y)$ and to the line $L(T_a(x),T_a(y))$, respectively.
\end{corollary}
\begin{proof}
By Theorem \ref{thm_vuo} and conformal invariance of the hyperbolic metric,
\begin{align*}
\frac{1}{\sqrt{1-m_1^2}}\,{\rm sh}\frac{h_{\B^n}(x,y)}{2}
={\rm sh}\frac{\rho_{\B^n}(x,y)}{2}
=\frac{1}{\sqrt{1-m_2^2}}{\rm sh}\frac{h_{\B^n}(T_a(x),T_a(y))}{2}\\
\end{align*}
from which the result follows.
\end{proof}

\begin{remark}
For $n=2$, the distance $m_1$ used in Corollary \ref{cor_dist} is as in \eqref{q_m1f} and, by inputting $T_a(x)$ and $T_a(y)$ into this formula \eqref{q_m1f}, we can solve that
\begin{align*}
m_2
=\frac{|{\Im}((x-a)(\overline{y}-\overline{a})(1-a\overline{x})({1-\overline{a}y}))|}{|(x-a)(1-a\overline{x})|1-a\overline{y}|^2-(y-a)(1-a\overline{y})|1-a\overline{x}|^2|}.
\end{align*}
\end{remark}

\begin{remark}
For all $t>0$ and $k\geq1$, we have $k{\rm sh}(t)\leq{\rm sh}(kt)$, so it follows from Corollary \ref{cor_dist} that
\begin{align*}
h_{\B^n}(T_a(x),T_a(y))
\leq\sqrt{\frac{1-m_2^2}{1-m_1^2}}\,h_{\B^n}(x,y),
\end{align*}
for all $x,y,a\in\B^n$ if $m_2\leq m_1$.    
\end{remark}

\begin{corollary}\label{cor_minxy}
For all $x,y,a\in\B^n$,
\begin{align*}
h_{\B^n}(T_a(x),T_a(y))
\leq2\,{\rm arsh}\left(\frac{1}{\sqrt{1-\min\{|x|,|y|,|x+y|/2\}^2}}\,{\rm sh}\frac{h_{\B^n}(x,y)}{2}\right).
\end{align*}
\end{corollary}
\begin{proof}
Follows from Corollary \ref{cor_dist}.    
\end{proof}

\begin{corollary}
For all $a\in\B^n\setminus\{0\}$ and $x,y\in\B^n(0,|a|/2)$,
\begin{align*}
h_{\B^n}(T_a(x),T_a(y))
\leq2\,{\rm arsh}\left(\frac{1}{\sqrt{1-|a|^2/4}}\,{\rm sh}\frac{h_{\B^n}(x,y)}{2}\right).
\end{align*}
\end{corollary}
\begin{proof}
Follows from Corollary \ref{cor_minxy}.    
\end{proof}

Numerical tests suggest that the following result holds.

\begin{conjecture}\label{conj_thb}
For all $x,y,a\in\B^n$,
\begin{align*}
\frac{h_{\B^n}(T_a(x),T_a(y))}{h_{\B^n}(x,y)}\leq1+|a|.
\end{align*}
\end{conjecture}

\begin{remark}
The distortion of several other hyperbolic-type metrics under this mapping has been researched in \cite{sch}, for which conjectures similar to Conjecture \ref{conj_thb} has been suggested, see for instance \cite[Conj. 1.6, p. 684]{chkv}.    
\end{remark}

We will yet present one corollary related to the distortion of the Hilbert metric under quasi-regular mappings. See \cite[pp. 289-288]{hkv} for the definition of $K$-quasiregular mappings. Define an increasing homeomorphism $\varphi_{K,2}:[0,1]\to[0,1]$, \cite[(9.13), p. 167]{hkv} as
\begin{align*}
\varphi_{K,2}(r)=\frac{1}{\gamma_2^{-1}(K\gamma_2(1\slash r))},\quad
0<r<1,\,K>0.
\end{align*}
Here, the function $\gamma_2:(1,\infty)\to (0,\infty)$ is a decreasing homeomorphism known as the \emph{Gr\"otzsch capacity}, and it has the following formula \cite[(7.18), p. 122]{hkv}
\begin{align*}
\gamma_2(1/r)=\frac{2\pi}{\mu(r)},\quad \mu(r)=\frac{\pi}{2}\frac{\K(\sqrt{1-r^2})}{\K(r)},\quad
\K(r)=\int^1_0 \frac{dx}{\sqrt{(1-x^2)(1-r^2x^2)}}
\end{align*}
with $0<r<1$. Define then a number \cite[(10.4), p. 203]{avv}
\begin{align*}
\lambda(K)=\left(\frac{\varphi_{K,2}(1/\sqrt{2})}{\varphi_{1/K,2}(1/\sqrt{2})}\right)^2.   
\end{align*}
Clearly, $\lambda(1)=1$ and, if $K>1$, we have $\lambda(K)<e^{\pi(K-1/K)}$ by \cite[Thm 10.35, p. 219]{avv}. 

\begin{theorem}\label{thm_kqus}
If $f:\B^2\to\B^2$ is a non-constant $K$-quasiregular mapping, then
\begin{align*}
{\rm sh}\frac{\rho_{\B^2}(f(x),f(y))}{2}\leq\lambda(K)^{1/2}\max\left\{\left({\rm sh}\frac{\rho_{\B^2}(x,y)}{2}\right)^K,\left({\rm sh}\frac{\rho_{\B^2}(x,y)}{2}\right)^{1/K}\right\}  
\end{align*}   
for all $x,y\in\B^2$.
\end{theorem}
\begin{proof}
According to the Schwarz lemma \cite[Thm 16.2(1), p. 300]{hkv}, 
\begin{align*}
{\rm th}\frac{\rho_{\B^2}(f(x),f(y))}{2}\leq\varphi_{K,2}\left({\rm th}\frac{\rho_{\B^2}(x,y)}{2}\right).    
\end{align*}
For all $u\geq0$, the hyperbolic functions fulfill the following identity:
\begin{align*}
{\rm sh}(u)=\frac{{\rm th}(u)}{\sqrt{1-{\rm th}^2(u)}}.    
\end{align*}
Consequently,
\begin{align*}
{\rm sh}\frac{\rho_{\B^2}(f(x),f(y))}{2}
=\frac{{\rm th}(\rho_{\B^2}(f(x),f(y))/2)}{\sqrt{1-{\rm th}^2(\rho_{\B^2}(f(x),f(y))/2)}}
\leq\frac{\varphi_{K,2}({\rm th}(\rho_{\B^2}(x,y)/2))}{\sqrt{1-\varphi_{K,2}({\rm th}(\rho_{\B^2}(x,y)/2))^2}}.   \end{align*}
Using the function $\eta_{K,2}$ defined in \cite[(10.3), p. 203]{avv}, we have
\begin{align*}
\frac{\varphi_{K,2}({\rm th}(\rho_{\B^2}(x,y)/2))}{\sqrt{1-\varphi_{K,2}({\rm th}(\rho_{\B^2}(x,y)/2))^2}}=\eta_{K,2}\left({\rm sh}^2\frac{\rho_{\B^2}(x,y)}{2}\right)^{1/2}.    
\end{align*}
By \cite[Thm 10.24, p. 214]{avv}, $\eta_{K,2}(t)\leq\lambda(K)\max\{t^K,t^{1/K}\}$. Thus,
\begin{align*}
{\rm sh}\frac{\rho_{\B^2}(f(x),f(y))}{2}&\leq\eta_{K,2}\left({\rm sh}^2\frac{\rho_{\B^2}(x,y)}{2}\right)^{1/2}\\
&\leq\lambda(K)^{1/2}\max\left\{\left({\rm sh}\frac{\rho_{\B^2}(x,y)}{2}\right)^K,\left({\rm sh}\frac{\rho_{\B^2}(x,y)}{2}\right)^{1/K}\right\}.   
\end{align*}
\end{proof}

\begin{corollary}
If $f:\B^2\to\B^2$ is a non-constant $K$-quasiregular mapping, then
\begin{align*}
{\rm sh}\frac{h_{\B^2}(f(x),f(y))}{2}\leq\lambda(K)^{1/2}\sqrt{1-m_3^2}\max\left\{\left(\frac{{\rm sh}(h_{\B^2}(x,y)/2)}{\sqrt{1-m_1^2}}\right)^K,\left(\frac{{\rm sh}(h_{\B^2}(x,y)/2)}{\sqrt{1-m_1^2}}\right)^{1/K}\right\}  
\end{align*}   
for all $x,y\in\B^2$, where $m_1$ and $m_3$ are the Euclidean distances from the origin to the line $L(x,y)$ and to the line $L(f(x),f(y))$, respectively. We have equality here if $K=1$ and $m_1=m_3$.
\end{corollary}
\begin{proof}
Follows from Theorems \ref{thm_vuo} and \ref{thm_kqus}.     
\end{proof}

\bibliographystyle{siamplain}

\end{document}